\documentclass[oneside,english]{amsart}
\usepackage[T1]{fontenc}
\usepackage[latin9]{inputenc}
\usepackage[letterpaper]{geometry}
\usepackage{amsthm}
\usepackage{amstext}
\usepackage{amssymb}
\usepackage{esint}

\makeatletter
\numberwithin{equation}{section}
\numberwithin{figure}{section}
\theoremstyle{plain}
\newtheorem{thm}{\protect\theoremname}

\makeatother

\usepackage{babel}
\providecommand{\theoremname}{Theorem}

\begin{document}

\title{Some nonlinear functions of Bernoulli and Euler umbr\ae}

\author{Christophe Vignat, Université d'Orsay}
\begin{abstract}
In a recent paper \cite{Yu}, Yi-Ping Yu has given some interesting
nonlinear moments of the Bernoulli umbra; the aim of this paper is
to show the probabilistic counterpart of these results and to extend
them to Bernoulli polynomials.
\end{abstract}
\maketitle

\section{Introduction}

In a recent rich contribution, Yi-Ping Yu gives several nonlinear
moments of the Bernoulli umbra $\mathfrak{B}$ defined by its generating
function\[
\exp\left(z\mathfrak{B}\right)=\frac{z}{\exp\left(z\right)-1},\,\,\vert z\vert<2\pi.
\]
This umbra is related to the Bernoulli numbers as\[
\mathfrak{B}^{n}=B_{n};
\]
for example\[
B_{0}=1;\,\, B_{1}=-\frac{1}{2};\,\, B_{2}=\frac{1}{6};\,\, B_{3}=0;\,\, B_{4}=-\frac{1}{30}.
\]
and all odd orders Bernoulli numbers except $B_{1}$ equal $0.$ 

Similarly, the Euler umbra $\mathfrak{E}$ is defined by the generating
function\[
\exp\left(z\mathfrak{E}\right)=\text{sech}\left(z\right)
\]
We generalize here these umbr\ae \,\,and define the Bernoulli umbra
$\mathfrak{B}\left(x\right)$ as\begin{equation}
\exp\left(z\mathfrak{B}\left(x\right)\right)=\frac{ze^{zx}}{e^{z}-1}\label{eq:expzB(x)}
\end{equation}
and the Euler umbra $\mathfrak{E}\left(x\right)$ as\[
\exp\left(z\mathfrak{E}\left(x\right)\right)=\frac{2e^{zx}}{e^{z}+1}.
\]
As a result, \[
\mathfrak{B}^{n}\left(x\right)=B_{n}\left(x\right)
\]
and \[
\mathfrak{E}^{n}\left(x\right)=E_{n}\left(x\right),
\]
respectively the Bernoulli and Euler polynomials of degree $n.$ 

The aim of this paper is to compute some nonlinear functions of these
umbr\ae \,\,as probabilistic nonlinear moments. In the following,
we denote the expectation operator \[
Eh\left(X\right)=\int h\left(x\right)f_{X}\left(x\right)dx
\]
where $f_{X}$ is the probability density function of the random variable
$X.$ We will use the following characterization of the Bernoulli
and Euler umbr\ae.
\begin{thm}
The Bernoulli umbra $\mathfrak{B}\left(x\right)$ satisfies, for all
admissible function $h,$\[
h\left(\mathfrak{B}\left(x\right)\right)=Eh\left(x-\frac{1}{2}+\imath L_{B}\right)
\]
where the random variable $L_{B}$ follows a logistic distribution,
with density\[
f_{L_{B}}\left(x\right)=\frac{\pi}{2}\text{sech}^{2}\left(\pi x\right),\,\, x\in\mathbb{R}.
\]
Accordingly, the Euler umbra $\mathfrak{E}\left(x\right)$ satisfies,
for all admissible function $h,$\[
h\left(\mathfrak{E}\left(x\right)\right)=Eh\left(x-\frac{1}{2}+\imath L_{E}\right)
\]
where the random variable $L_{E}$ follows the hyperbolic secant distribution
\[
f_{L_{E}}\left(x\right)=\text{sech}\left(\pi x\right).
\]
\end{thm}
\begin{proof}
Since \[
\exp\left(\imath t\mathfrak{B}\left(x\right)\right)=E\exp\left(it\left(x-\frac{1}{2}+\imath L_{B}\right)\right),
\]
by identification with (\ref{eq:expzB(x)}), the random variable $L_{B}$
has characteristic function\[
E\left(e^{\imath tL_{B}}\right)=\frac{\frac{t}{2}}{\sinh\left(\frac{t}{2}\right)}.
\]
But from \cite[1.9.2]{Bateman}\[
\int_{0}^{+\infty}\text{sech}^{2}\left(ax\right)\cos\left(xt\right)dx=\frac{\pi t}{2a^{2}}\text{csch}\left(\frac{\pi t}{2a}\right)
\]
so that, with $a=\pi,$ the density of $L_{B}$ is \[
f_{L_{B}}\left(x\right)=\frac{\pi}{2}\text{sech}^{2}\left(\pi x\right),
\]
which is a logistic density.

Accordingly, the characteristic function of the random variable $L_{E}$
is\[
Ee^{\imath L_{E}t}=\text{sech}\left(\frac{t}{2}\right).
\]
From \cite[1.9.1]{Bateman}, \[
\int_{0}^{+\infty}\text{sech}\left(ax\right)\cos\left(xt\right)dx=\frac{\pi}{2a}\text{sech}\left(\frac{\pi}{2a}t\right)
\]
so that, with $a=\pi,$ the density of $L_{E}$ is\[
f_{L_{E}}\left(x\right)=\text{sech}\left(\pi x\right).
\]
Thus $\pi L_{0}$ follows an hyperbolic secant distribution.
\end{proof}
As a consequence, the Bernoulli polynomials read \begin{equation}
B_{n}\left(x\right)=\mathfrak{B}\left(x\right)^{n}=E\left(x-\frac{1}{2}+\imath L_{B}\right)^{n}\label{eq:BnU}
\end{equation}
and the Bernoulli numbers\begin{eqnarray*}
B_{n} & =\mathfrak{B}{}^{n}=\mathfrak{B}\left(0\right)^{n}= & E\left(-\frac{1}{2}+\imath L_{B}\right)^{n},\,\, n\ge0.
\end{eqnarray*}
 Similarly, the Euler polynomials read \[
E_{n}\left(x\right)=\mathfrak{E}\left(x\right)^{n}=E\left(x-\frac{1}{2}+\imath L_{E}\right)^{n}
\]
and the Euler numbers\[
E_{n}=2^{n}\mathfrak{E}\left(\frac{1}{2}\right)^{n}=2^{n}E\left(\imath L_{E}\right)^{n}.
\]

We note from \cite[p. 471]{Devroye} that the random variable $L_{B}$
can also be obtained as \[
L_{B}=\frac{1}{2\pi}\log\frac{U}{1-U}=\frac{1}{2\pi}\log\frac{E_{1}}{E_{2}}
\]
where U is uniformly distributed on $\left[-1,+1\right]$, $E_{1}$
and $E_{2}$ are independent with exponential distribution $f_{E}\left(x\right)=\exp\left(-x\right),\,\, x\in\left[0,+\infty\right[$
and equality is in the sense of distributions. As for the random variable
$L_{E}$, from \cite{Devroye}, it can be obtained as \begin{equation}
L_{0}=\frac{1}{\pi}\log\vert C\vert=\frac{1}{\pi}\left(\log\vert N_{1}\vert-\log\vert N_{2}\vert\right)\label{eq:L0}
\end{equation}
where $C$ is Cauchy distributed and $N_{1}$ and $N_{2}$ are two
independent standard Gaussian random variables.

\section{the moment $\log\mathfrak{B}\left(x\right)$}

We compute\textbf{\emph{\[
\log\mathfrak{B}\left(x\right)=E\log\left(x-\frac{1}{2}+\imath L_{B}\right)
\]
}}which, by symmetry, is equal to\[
\frac{1}{2}E\log\left(\left(x-\frac{1}{2}\right)^{2}+L_{B}^{2}\right)=\log\left|x-\frac{1}{2}\right|+\frac{1}{2}E\log\left(1+\frac{L_{B}^{2}}{\left(x-\frac{1}{2}\right)^{2}}\right)
\]
but from \cite[2.6.30.2]{Prudnikov}\[
\int_{0}^{+\infty}\frac{\log\left(1+bz^{2}\right)}{\sinh^{2}cz}dz\overset{d}{=}h\left(b,c\right)=\frac{2}{c}\left(\log\frac{c}{\pi\sqrt{b}}-\psi\left(\frac{c}{\pi\sqrt{b}}\right)\right).
\]
Thus, by bisection of the angle $2\pi z,$ \[
\int_{0}^{+\infty}\frac{\log\left(1+bz^{2}\right)}{\sinh^{2}2\pi z}dz=\frac{1}{4}\int_{0}^{+\infty}\frac{\log\left(1+bz^{2}\right)}{\sinh^{2}\pi z\cosh^{2}\pi z}dz=\frac{1}{4}\int_{0}^{+\infty}\frac{\log\left(1+bz^{2}\right)}{\cosh^{2}\pi z}\left(\frac{\cosh^{2}\pi z}{\sinh^{2}\pi z}-1\right)dz
\]
so that \begin{eqnarray*}
\frac{\pi}{2}\int_{-\infty}^{+\infty}\frac{\log\left(1+bz^{2}\right)}{\cosh^{2}\pi z}dz & = & \pi\left(h\left(b,\pi\right)-4h\left(b,2\pi\right)\right).
\end{eqnarray*}
We deduce, with $b=\left(x-\frac{1}{2}\right)^{-2},$

\begin{eqnarray*}
\frac{1}{2}E\log\left(1+\frac{L_{B}^{2}}{\left(x-\frac{1}{2}\right)^{2}}\right) & = & \frac{\pi}{2}\left(h\left(b,\pi\right)-4h\left(b,2\pi\right)\right)=\frac{\pi}{2}\left(\frac{2}{\pi}\left(\log\frac{1}{\sqrt{b}}-\psi\left(\frac{1}{\sqrt{b}}\right)\right)-\frac{4}{\pi}\left(\log\frac{2}{\sqrt{b}}-\psi\left(\frac{2}{\sqrt{b}}\right)\right)\right)\\
 & = & \log\frac{1}{\sqrt{b}}-2\log\frac{2}{\sqrt{b}}-\psi\left(\frac{1}{\sqrt{b}}\right)+2\psi\left(\frac{2}{\sqrt{b}}\right)
\end{eqnarray*}
and, using the identity\[
\psi\left(2z\right)=\frac{1}{2}\psi\left(z\right)+\frac{1}{2}\psi\left(z+\frac{1}{2}\right)+\log2,
\]
we obtain after simplification\[
E\log\left(\imath L_{B}+x-\frac{1}{2}\right)=\log\frac{1}{\sqrt{b}}+E\log\left(1+bL_{B}^{2}\right)=\psi\left(\frac{1}{2}+\vert x-\frac{1}{2}\vert\right).
\]
For $x=1,$ we recover the result by Y.-P. Yu, namely\[
E\log\left(\frac{1}{2}+\imath L_{B}\right)=\psi\left(1\right)=-\gamma.
\]

\section{the moment $\log\mathfrak{E}\left(x\right)$}

This moment can be obtained according to the same approach, namely,
again with $b=\left(x-\frac{1}{2}\right)^{-2},$\[
\log\mathfrak{E}\left(x\right)=\log\frac{1}{\sqrt{b}}+\frac{1}{2}E\log\left(1+bL_{E}^{2}\right)
\]
where the latter expectation is now computed using \cite[2.6.30.1]{Prudnikov}
as\[
\int_{0}^{+\infty}\frac{\log\left(1+bz^{2}\right)}{\cosh\left(\pi z\right)}dz=2\log\frac{\Gamma\left(\frac{3}{4}+\frac{1}{2\sqrt{b}}\right)}{\Gamma\left(\frac{1}{4}+\frac{1}{2\sqrt{b}}\right)}-\log\frac{1}{2\sqrt{b}}
\]
so that\[
\log\mathfrak{E}\left(x\right)=\log2\frac{\Gamma^{2}\left(\frac{3}{4}+\frac{1}{2}\vert x-\frac{1}{2}\vert\right)}{\Gamma^{2}\left(\frac{1}{4}+\frac{1}{2}\vert x-\frac{1}{2}\vert\right)}.
\]

\section{the moments $\mathfrak{B}^{-k}\left(x\right)$ and $\mathfrak{E}^{-k}\left(x\right)$}

By derivation of the preceding results, we deduce\[
\mathfrak{B^{-1}}\left(x\right)=E\left(x-\frac{1}{2}+\imath L_{B}\right)^{-1}=\frac{d}{dx}\log\mathfrak{B}\left(x\right)
\]
so that we have\[
\mathfrak{B}^{-1}\left(x\right)=\begin{cases}
\psi'\left(x\right), & x>\frac{1}{2}\\
-\psi'\left(-x+1\right), & x<\frac{1}{2}\\
0 & x=\frac{1}{2}
\end{cases}
\]
and we remark that $\mathfrak{B}^{-1}\left(x\right)$ is not continuous
in $x=\frac{1}{2}.$ Since moreover for any integer $k\ge1$\[
E\left(x-\frac{1}{2}+\imath L_{B}\right)^{-k}=\frac{\left(-1\right)^{k-1}}{\left(k-1\right)!}\frac{d^{k-1}}{dx^{k-1}}E\left(x-\frac{1}{2}+\imath L_{B}\right)^{-1}
\]
we deduce\[
\mathfrak{B}^{-k}\left(x\right)=E\left(x-\frac{1}{2}+\imath L_{B}\right)^{-k}=\begin{cases}
\frac{\left(-1\right)^{k-1}}{\left(k-1\right)!}\psi^{\left(k\right)}\left(x\right), & x>\frac{1}{2}\\
-\frac{1}{\left(k-1\right)!}\psi^{\left(k\right)}\left(-x+1\right) & x<\frac{1}{2}
\end{cases}
\]
and in a particular case $x=1$, since $\psi^{\left(k\right)}\left(1\right)=\left(-1\right)^{k+1}k!\zeta\left(k+1\right),$\[
\mathfrak{B}^{-k}\left(1\right)=E\left(\frac{1}{2}+\imath L_{B}\right)^{-k}=k.\zeta\left(k+1\right).
\]
In the Euler case, we have\[
\mathfrak{E}^{-1}\left(x\right)=\frac{d}{dx}\log\mathfrak{E}\left(x\right)=\begin{cases}
\psi\left(\frac{x+1}{2}\right)-\psi\left(\frac{x}{2}\right) & x>\frac{1}{2}\\
\psi\left(\frac{1-x}{2}\right)-\psi\left(1-\frac{x}{2}\right) & x<\frac{1}{2}\\
0 & x=\frac{1}{2}
\end{cases}
\]
and $\mathfrak{E}^{-1}\left(x\right)$ is not continuous in $x=\frac{1}{2}$. 

More generally, for any integer $k\ge1,$\begin{eqnarray*}
\mathfrak{E}^{-k}\left(x\right) & = & \begin{cases}
\frac{\left(-\frac{1}{2}\right)^{k-1}}{\left(k-1\right)!}\left(\psi^{\left(k-1\right)}\left(\frac{x+1}{2}\right)-\psi^{\left(k-1\right)}\left(\frac{x}{2}\right)\right), & x>-\frac{1}{2}\\
\frac{\left(\frac{1}{2}\right)^{k-1}}{\left(k-1\right)!}\left(\psi^{\left(k-1\right)}\left(\frac{1-x}{2}\right)-\psi^{\left(k-1\right)}\left(1-\frac{x}{2}\right)\right), & x<-\frac{1}{2}
\end{cases}
\end{eqnarray*}

\section{the moment $\log\sin\frac{\pi\mathfrak{B}}{2}$}

This moment can be easily computed from the moment representation
as follows\begin{eqnarray*}
\log\sin\frac{\pi\mathfrak{B}}{2} & = & E\log\sin\frac{\pi}{2}\left(-\frac{1}{2}+\imath L_{B}\right)=E\log\sin\left(-\frac{\pi}{4}-\imath\frac{\pi L_{B}}{2}\right)=E\log\sin\left(-\frac{\pi}{4}+\imath\frac{\pi L_{B}}{2}\right)\\
 & = & \frac{1}{2}E\log\sin\left(-\frac{\pi}{4}-\imath\frac{\pi L_{B}}{2}\right)\sin\left(-\frac{\pi}{4}+\imath\frac{\pi L_{B}}{2}\right)
\end{eqnarray*}
and expanding the product of sines we obtain\[
\frac{1}{2}E\log\left(\frac{1}{2}\cos\left(-\frac{\pi}{2}\right)+\frac{1}{2}\cos\left(\imath\pi L_{B}\right)\right)=-\frac{1}{2}\log2+\frac{1}{2}E\log\cosh\left(\pi L_{B}\right).
\]
But since \[
\pi L_{B}=\frac{1}{2}\log\frac{U}{1-U},
\]
we deduce\[
\cosh\left(\pi L_{B}\right)=\frac{1}{2\sqrt{U\left(1-U\right)}}
\]
so that, with $E\log U=-1,$ we deduce\[
E\log\cosh\left(\pi L_{B}\right)=-\log2+1
\]
and the result \[
\log\sin\frac{\pi\mathfrak{B}}{2}=\frac{1}{2}-\log2
\]
follows.

\section{the Pochhammer $\left(\mathfrak{B}\left(x\right)\right)_{n}$}

The Pochhammer symbol \[
\left(\mathfrak{B}+1\right)_{n}=\frac{\Gamma\left(\mathfrak{B}+n+1\right)}{\Gamma\left(\mathfrak{B}+1\right)}
\]
has been evaluated in \cite[p.149]{N=0000F6rlund} as\[
\left(\mathfrak{B}+1\right)_{n}=\frac{n!}{\left(n+1\right)}.
\]
We use the {}``intuitive argument'' suggested by Carlitz \cite{Carlitz}
to compute its polynomial version $\left(\mathfrak{B}\left(x\right)\right)_{n}$
as follows: a generating function of $\left(\mathfrak{B}\left(x\right)\right)_{n}$
is\begin{eqnarray*}
\varphi\left(x,t\right) & = & \sum_{n=0}^{+\infty}\left(\mathfrak{B}\left(x\right)\right)_{n}\frac{t^{n}}{n!}=E\exp\left(-\left(x-\frac{1}{2}+\imath L_{B}\right)\log\left(1-t\right)\right)\\
 & = & \left(1-t\right)^{-\left(x-\frac{1}{2}\right)}E\exp\left(-\imath L_{B}\log\left(1-t\right)\right)
\end{eqnarray*}
with the characteristic function for the logistic density \[
E\exp\left(iL_{B}u\right)=\frac{\frac{u}{2}}{\sinh\left(\frac{u}{2}\right)}
\]
so that \[
\varphi\left(x,t\right)=\left(1-t\right)^{-\left(x-\frac{1}{2}\right)}\frac{\frac{1}{2}\log\left(1-t\right)}{\sinh\left(\frac{\log\left(1-t\right)}{2}\right)}=-\left(1-t\right)^{-\left(x-1\right)}\frac{\log\left(1-t\right)}{t}
\]
This term is identified as the derivative\[
\frac{d}{dx}\frac{\left(1-t\right)^{-\left(x-1\right)}}{t}=\frac{d}{dx}\sum_{n=0}^{+\infty}\frac{t^{n-1}}{n!}\left(x-1\right)_{n}=\sum_{n=0}^{+\infty}\frac{t^{n-1}}{n!}\frac{d}{dx}\left(x-1\right)_{n}
\]
with\[
\frac{d}{dx}\left(x-1\right)_{n}=\left(x-1\right)_{n}\left(\psi\left(x+n-1\right)-\psi\left(x-1\right)\right)
\]
so that the coefficient of $\frac{t^{n}}{n!}$ in $\varphi\left(x,t\right)$
is\[
\left(\mathfrak{B}\left(x\right)\right)_{n}=\frac{\left(x-1\right)_{n+1}}{n+1}\left(\psi\left(x+n\right)-\psi\left(x-1\right)\right).
\]
We recover the result by Nörlund by taking the limit case $x\to1$
which is $\frac{n!}{n+1}.$

\section{the Pochhammer $\left(\mathfrak{E}\left(x\right)\right)_{n}$}

We use the same approach to compute the Pochhammer symbol of the Euler
polynomial umbra; the generating function reads\begin{eqnarray*}
\varphi\left(x,t\right) & = & \sum_{n=0}^{+\infty}\left(\mathfrak{E}\left(x\right)\right)_{n}\frac{t^{n}}{n!}=E\exp\left(-\left(x-\frac{1}{2}+\imath L_{E}\right)\log\left(1-t\right)\right)\\
 & = & \left(1-t\right)^{-\left(x-\frac{1}{2}\right)}E\exp\left(-\imath L_{E}\log\left(1-t\right)\right)
\end{eqnarray*}
with the characteristic function of the hyperbolic secant distribution\[
Ee^{\imath L_{E}t}=\text{sech}\left(\frac{t}{2}\right)
\]
so that\[
E\exp\left(\imath L_{E}\log\left(1-t\right)\right)=\text{sech}\left(\frac{1}{2}\log\left(1-t\right)\right)=\frac{\sqrt{1-t}}{1-\frac{t}{2}}
\]
and\[
\varphi\left(x,t\right)=\frac{1}{\left(1-t\right)^{x-1}\left(1-\frac{t}{2}\right)}=\sum_{n=0}^{+\infty}\frac{t^{n}}{n!}\frac{n!}{2^{n}}\sum_{k=0}^{n}\frac{\left(x-1\right)_{k}}{k!}2^{k}
\]
so that \[
\left(\mathfrak{E}\left(x\right)\right)_{n}=\frac{n!}{2^{n}}\sum_{k=0}^{n}\frac{\left(x-1\right)_{k}}{k!}2^{k}.
\]

\end{document}